\documentclass[11pt]{article}
\usepackage{graphicx}
\usepackage{placeins}
\usepackage{amsmath,amsthm,latexsym,amssymb,amsfonts,latexsym}
\usepackage{bm}
\usepackage[noadjust]{cite}
\usepackage{makecell}

\newtheorem{theorem}{Theorem}[section]

\newtheorem{lemma}{Lemma}[section]
\newtheorem{corollary}{Corollary}[section]

\numberwithin{equation}{section}
\newcommand\numberthis{\addtocounter{equation}{1}\tag{\theequation}}
\newcommand*{\Resize}[2]{\resizebox{#1}{!}{$#2$}}
\newcommand{\spec}{\operatorname{Spec}}

\newcommand{\de}{\operatorname{det}}
\newcommand{\sgn}{\operatorname{sgn}}
\newcommand{\mi}{\operatorname{min}}
\newcommand{\ma}{\operatorname{max}}
\newcommand{\trace}{\operatorname{trace}}
\def \T {{\mathbb T}}

\def \x {{\mathbf x}}
\def \y {{\mathbf y}}

\def \1 {{\mathbf 1}}
\def \i {{\mathbf i}}
\def \omeg {{\bm{\omega}}}

\usepackage{caption, subcaption}
\usepackage[width=150mm,top=25mm,bottom=25mm]{geometry}
\usepackage{mathtools}
\usepackage{enumerate}

\usepackage{cite}
\date{}
\begin{document}
\title{\textbf{On spectra of Hermitian Randi\'c matrix of second kind }}
	\date{}
	\author{A. Bharali$^1$\footnote{Corresponding author: a.bharali@dibru.ac.in}, Bikash Bhattacharjya$^2$, Sumanta Borah$^3$, Idweep Jyoti Gogoi$^4$\\
	Department of Mathematics, Dibrugarh University, Assam, India$^{1,\;3,\;4}$\\
	Department of Mathematics, Indian Institute of Technology Guwahati, India$^2$}
	\maketitle

\begin{abstract}
Let $X$ be a mixed graph and $\omega=\frac{1+\i \sqrt{3}}{2}$. We write $i\rightarrow j$, if there is an oriented edge from a vertex $v_i$ to another vertex $v_j$, and $i\sim j$ for an un-oriented edge between the vertices $v_i$ and $v_j$. The degree of a vertex $v_i$ is denoted by $d_i$. We propose the Hermitian Randi\'c matrix of second kind $R^\omeg(X)\coloneqq(R^\omeg_{ij})$, where $R^\omeg_{ij}=\frac{1}{\sqrt{d_id_j}}$ if $i \sim j$, $R^\omeg_{ij}= \frac{\omega}{\sqrt{d_id_j}}$ and  $R^\omeg_{ji}= \frac{\overline{\omega}}{\sqrt{d_id_j}}$ if  $i\rightarrow j$, and 0 otherwise. In this paper, we investigate some spectral features of this novel Hermitian matrix and study a few properties like positiveness, bipartiteness, edge-interlacing etc. We also compute the characteristic polynomial for this new matrix and obtain some upper and lower bounds for the eigenvalues and the energy of this matrix. \\
\itshape \bf {Keywords:} \normalfont Mixed graph; Hermitian adjacency matrix; Hermitian Randi\'c matrix; graph energy\normalfont\\
\itshape \bf {AMS Classifications (2010):} \normalfont 05C50; 05C09; 05C31
\end{abstract}
\section{Introduction}
There has been an upsurge of studies related to spectral properties of graph theoretical matrices. Investigation of these properties play a vital role in analysing some properties of networks. In recent times, the extensions of spectral theory of un-oriented networks to mixed networks is a popular topic. In comparison to the un-oriented networks, the mixed networks are much better to model the real world problems. However, we see that many graph matrices for mixed networks appear to be non-symmetric, losing the property that eigenvalues are real.

Recently, many researcher studied the spectral properties of adjacency matrix, Laplacian matrix, normalized Laplacian matrix etc. of mixed networks by incorporating modified versions of these matrices. For details, see \cite{ars,b1,gm,ylq,ywgq}. In 2015, Yu and Qu \cite{yq} described some notable works on Hermitian Laplacian matrix of mixed graphs. In the same year, Liu and Li \cite{ll} studied some properties of Hermitian adjacency matrix. They also determined some bounds for energies of mixed graphs. Some similar works on Randi\'{c} matrix was done by Lu et al. \cite{lwz} in 2017. Yu et al. \cite{ydsj} in 2019, defined the Hermitian normalized Laplacian matrix and studied some spectral properties for mixed networks. In 2020, B. Mohar \cite{m1} introduced a new modified Hermitian matrix that seems more natural. Some relevant notable works can be found in \cite{gs,ls,lw,ly,m2,r,sk}.

The energy levels of $\pi$-electrons in conjugated hydrocarbons in molecular orbital are strongly related in spectral graph theory. In 1978, Gutman \cite{g0} developed the notion of graph energy based on eigenvalues of a graph. Since then it plays an important role in chemical graph theory. Later many variants of graph energy, based on different matrices other than the adjacency matrix, were proposed as a consequence of the success of the notion of graph energy, for details see \cite{bggc, bmgd, c, fa, kkp, lw}. In 2010, Bozkurt et al. \cite{bggc} proposed the Randi\'c energy of graph as the sum of the absolute values of the eigenvalues of the Randi\'c matrix. In 2017, Lu et al. \cite{lwz} introduced Hermitian Randi\'c matrix for mixed graphs, which was based on the Hermitian adjacency matrix proposed by Guo and Mohar \cite{gm}. They also investigated the energy for this matrix. In this paper, we define the Hermitian Randi\'c matrix of second kind of a mixed graph,  and study some properties of its eigenvalues and energy.

\section{Basic Definitions}
Throughout the paper, we consider connected simple graph with at least two vertices. Let $X$ be an un-oriented graph. We denote an edge of $X$ between the vertices $v_i$ and $v_j$  by $e_{ij}$. Note that the edge $e_{ij}$ can be assigned two orientations. An oriented edge originating at $v_i$ an terminating at $v_j$ is denoted by $\overrightarrow{e_{ij}}$. For each edge $e_{ij}\in E(X)$, there is a pair of oriented edges $\overrightarrow{e_{ij}}$ and $\overrightarrow{e_{ji}}$. The collection $\overrightarrow{E}(X)\coloneqq \{\overrightarrow{e_{ij}}, \overrightarrow{e_{ji}}: e_{ij} \in E(X) \}$ is the oriented edge set associated with $X$. Note that each edge of an un-oriented graph is of the form $e_{ij}$. The set $\overrightarrow{E}(X)$ is the collection of all possible oriented edges of an un-oriented graph $X$.

A graph $X$ is said to be mixed if it has both possibilities of edges that are oriented and un-oriented. If $X$ is a mixed graph, then at most one of $e_{ij}, \overrightarrow{e_{ij}}$ and $\overrightarrow{e_{ji}}$ can be in $E(X)$. We write $i\rightarrow j$, if there is an oriented edge from vertex $v_i$ to vertex $v_j$, and $i\sim j$ for an un-oriented edge between the vertices $v_i$ and $v_j$. The graph  $X_U$ obtained from a mixed graph $X$ by replacing each of the oriented edge of $X$ by the corresponding un-oriented edge is called the {\itshape underlying graph} of $X$. A \emph{cycle} in a mixed graph is a cycle in its underlying graph. A cycle is even or odd according as its order is even or odd.

A \emph{gain graph} or $\T$-gain graph is a triplet $\Phi\coloneqq (X,\T, \varphi)$, where $X$ is an un-oriented graph, $\T=\{z\in\mathbb{C}: |z|=1\}$, and $\varphi: \overrightarrow{E}(X)\rightarrow \T$ is a function satisfying $\varphi(\overrightarrow{e_{ij}})=\varphi(\overrightarrow{e_{ji}})^{-1}$ for each $e_{ij} \in E(X)$. The function $\varphi$ is called the {\itshape gain function} of $(X,\T,\varphi)$. For simplicity, we use $\Phi\coloneqq (X, \varphi)$ to denote a $\T$-gain graph instead of $\Phi\coloneqq (X, \T, \varphi)$. For a $\T$-gain graph $\Phi\coloneqq (X,\T, \varphi)$,  the $\T$-gain graph $-\Phi$ is defined by $-\Phi\coloneqq (X, -\varphi)$. In \cite{r}, Reff proposed the notion of the adjacency matrix $A(\Phi)\coloneqq(a_{ij})$ of a $\T$-gain graph, where
 $$
a_{ij}=\left\{
  \begin{array}{ll}
    \varphi(\overrightarrow{e_{ij}}) & \hbox{if $i \rightarrow j$} \\
    0 & \hbox{otherwise.}
  \end{array}
\right.
$$
It is clear that $ A(\Phi)$ is Hermitian. Thus the eigenvalues of this matrix are real. If $\varphi(\overrightarrow{e_{ij}})=1$ for all $\overrightarrow{e_{ij}}$, then we have $ A(\Phi)= A(X)$, where $A(X)$ is the usual $(0,1)$-adjacency matrix of the graph $X$. Thus one can assume a graph $X$ as a $\T$-gain graph $(X,\mathbf 1)$, where $\mathbf 1$ is the function that assign 1 to each edge of $\overrightarrow{E}(X)$. 
A \emph {switching function} $\zeta$ of $X$ is a function from $V(X)$ to $\T$, that is, $\zeta:V(X)\rightarrow \T$. Two gain graphs $\Phi_1\coloneqq (X,\varphi_1)$ and $\Phi_2\coloneqq (X,\varphi_2)$ are said to be \emph{switching equivalent}, written $\Phi_1\sim \Phi_2$, if there exists a switching function  $\zeta:V(X)\rightarrow \T$ such that
$\varphi_2(\overrightarrow{e_{ij}})=\zeta(i)^{-1}\varphi_1(\overrightarrow{e_{ij}})\zeta(j).$

It is clear from the definition that the gain graphs $\Phi_1$ and $\Phi_2$ are switching equivalent if and only if there is a diagonal matrix $D_\zeta$, where the diagonal entries come from $\T$, such that
$$ A(\Phi_2)=D_\zeta^{-1} A(\Phi_1)D_\zeta. $$

In 2015, Guo and Mohar \cite{gm} introduced a Hermitian adjacency matrix $H(X)$ of a mixed graph $X$, where the $ij$-th entry is $\i, -\i$ or 1 according as $\overrightarrow{e_{ij}}\in E(X)$, $\overrightarrow{e_{ji}}\in E(X)$ or $e_{ij}\in E(X)$ respectively, and 0 otherwise. Here $\i=\sqrt{-1}$. This matrix has numerous appealing characteristics, including real eigenvalues and the interlacing theorem for mixed graphs etc.

Later in 2020, Mohar \cite{m1} put forward a new Hermitian adjacency matrix $H^{\omega}(X)\coloneqq(h_{ij})$ of a mixed graph $X$, which is referred as Hermitian matrix of second kind, where
$$
h_{ij}=\left\{
  \begin{array}{ll}
    1 & \hbox{if $i \sim j$} \\
    \omega & \hbox{if $i \rightarrow j$} \\
    \overline{\omega} & \hbox{if $j\rightarrow i$} \\
    0 & \hbox{otherwise.}
  \end{array}
\right.
$$
Here $\omega\coloneqq\frac{1+\mathbf{i}\sqrt{3}}{2}$, a primitive sixth root of unity and $\overline{\omega}\coloneqq\frac{1-\mathbf{i}\sqrt{3}}{2}$. For a mixed graph $X$, let $(X_U,\omeg)$ represent the $\T$-gain graph with the gain function $\omeg:\overrightarrow{E}(X_U)\rightarrow \{1, \omega, \overline{\omega}\}$, where
$$
\omeg(\overrightarrow{e_{ij}})=\left\{
  \begin{array}{ll}
    1 & \hbox{if $e_{ij} \in E(X)$} \\
    \omega & \hbox{if $\overrightarrow{e_{ij}}\in E(X)$}\\
    \overline{\omega} & \hbox{if $\overrightarrow{e_{ji}}\in E(X)$}.
  \end{array}
\right.
$$
Note that $H^{\omega}(X)=A(\Phi)$ for $\Phi=(X_U, \omeg)$.

With the growing popularity of these Hermitian matrices, the idea of investigating spectral properties of mixed networks based on other graph matrices is also evolved. The matrix $D^{-1/2}A(X)D^{-1/2}$ is called the normalized adjacency matrix of an un-oriented graph $X$, where $D=\text{diag}(d_1,\ldots,d_n)$ and $d_i$ denote the degree of the vertex $v_i$ for $i\in\{1,\ldots,n\}$. This matrix is popularly known as the Randi\'c matrix $R(X)$.  

The matrix $D^{-1/2}H(X)D^{-1/2}$ is called the Hermitian Randi\'c matrix, denoted $R_H(X)$, of a mixed graph $X$. Similarly, the matrix  $D^{-1/2}H^{\omega}(X)D^{-1/2}$ is called the Hermitian Randi\'c matrix of second kind, denoted $R^{\omega}(X)$, of a mixed graph $X$. 

If $R^\omeg(X)\coloneqq (R_{ij}^\omeg)$, we find that
$$
R_{ij}^\omeg=\left\{
  \begin{array}{ll}
    \frac{1}{\sqrt{d_id_j}} & \hbox{if $i \sim j$} \\
    \frac{\omeg}{\sqrt{d_id_j}} & \hbox{if $i \rightarrow j$} \\
    \frac{\overline{\omeg}}{\sqrt{d_id_j}} & \hbox{if $j \rightarrow i$} \\
    \;\;\;0 & \hbox{otherwise.}
 \end{array}
\right.
$$
Clearly, $R^\omeg(X)$ is Hermitian. Let $L^\omeg(X)=D-H^\omeg(X)$ and $\mathfrak{L}^\omeg(X)=D^{-1/2}L^\omeg(X) D^{-1/2}$. The matrices $L^\omeg(X)$ and $\mathfrak{L}^\omeg(X)$ are known as the Hermitian Laplacian matrix of second kind and the normalized Hermitian Laplacian matrix of second kind of $X$, respectively. It is clear that $R^\omeg(X)=I-\mathfrak{L}^\omeg(X)$, where $I$ is the identity matrix of appropriate order.

The Randi\'c matrix $R(\Phi)$ of a $\T$-gain graph $\Phi$ is defined by $R(\Phi)\coloneqq D^{-1/2}A(\Phi)D^{-1/2}$. Similarly, the matrix $I-R(\Phi)$ is called the normalized Laplacian matrix of a $\T$-gain graph $\Phi$. If $X$ is a mixed graph, then we see that $R^\omeg(X)=R(\Phi)$, where $\Phi=(X_U, \omeg)$.

A walk (or path) in a mixed graph is a walk (or path) in its underlying graph. The value of a walk $W\coloneqq v_{i_1}v_{i_2}\cdots v_{i_\ell}$ in a mixed graph $X$ is defined as $R^\omeg_{i_{1}i_2}R^\omeg_{i_{2}i_3}\cdots R^\omeg_{i_{\ell-1}i_\ell}$. A walk is called positive or negative according as its value is positive or negative, respectively. An acyclic mixed graph is defined to be positive. A non-acyclic mixed graph is positive or negative according as each of its mixed cycle is positive or negative, respectively. A mixed graph $X$ is called an \emph{elementary graph} if each component of $X$ is an edge or a cycle. 

For the easy reference of the reader, we list the various notations and names of matrices associated with an un-oriented graph, mixed graph or a gain graph in Table~\ref{table1}, Table~\ref{table2}, Table~\ref{table3} and Table~\ref{table4}. Here $D$ denotes the degree matrix of an un-oriented graph, or the degree matrix of the corresponding underlying graph in case of a mixed graph or gain graph. Further, $I$ denotes the identity matrix of appropriate order.

\begin{table}
\begin{center}
\caption{$X$ is an un-oriented graph}
\Resize{15cm}{\begin{tabular}{| c | c | c  | c | }
	\hline
	 \makecell{adjacency \\ matrix} & \makecell{Randi\'c matrix \\ $R(X)$}  & \makecell{Laplacian matrix \\ $L(X)$} & \makecell{normalized Laplacian \\ matrix $\mathfrak{L}(X)$ }  \\
	\hline
	$A(X)$ & $R(X)=D^{-1/2}A(X)D^{-1/2}$  & $L(X)=D-A(X)$ & 
	\makecell{$\mathfrak{L}(X)=D^{-1/2}L(X)D^{-1/2}$ \\ $=I-R(X)~~$ } \\
	\hline
\end{tabular}}
\label{table1}
\end{center}
\end{table}


\begin{table}
\begin{center}
\caption{$X$ is a mixed graph}
\Resize{15cm}{\begin{tabular}{| c | c | c  | c | }
	\hline
	 \makecell{Hermitian \\ adjacency \\ matrix} & \makecell{Hermitian Randi\'c \\ matrix  $R_H(X)$}  & \makecell{Hermitian Laplacian \\matrix  $L_H(X)$} & \makecell{normalized Hermitian \\ Laplacian matrix $\mathfrak{L}_H(X)$ }  \\
	\hline
	$H(X)$ & $R_H(X)=D^{-1/2}H(X)D^{-1/2}$  & $L_H(X)=D-H(X)$ & 
	\makecell{$\mathfrak{L}_H(X)=D^{-1/2}L_H(X)D^{-1/2}$ \\ $=I-R_H(X)~$ } \\
	\hline
\end{tabular}}
\label{table2}
\end{center}
\end{table}

\begin{table}
\begin{center}
\caption{$X$ is a mixed graph}
\Resize{15cm}{\begin{tabular}{| c | c | c  | c | }
	\hline
	 \makecell{Hermitian \\ adjacency \\ matrix of \\2nd kind} & \makecell{Hermitian Randi\'c \\ matrix  of 2nd kind\\  $R^\omeg(X)$}  & \makecell{Hermitian Laplacian \\matrix of 2nd kind\\  $L^\omeg(X)$} & \makecell{normalized Hermitian \\ Laplacian matrix \\ of 2nd kind $\mathfrak{L}^\omeg(X)$ }  \\
	\hline
	$H^\omeg(X)$ & $R^\omeg(X)=D^{-1/2}H^\omeg(X)D^{-1/2}$  & $L^\omeg(X)=D-H^\omeg(X)$ & 
	\makecell{$\mathfrak{L}^\omeg(X)=D^{-1/2}L^\omeg(X)D^{-1/2}$ \\ $=I-R^\omeg(X)~$ } \\
	\hline
\end{tabular}}
\label{table3}
\end{center}
\end{table}
\begin{table}
\begin{center}
\caption{$\Phi$ is a gain graph}
\Resize{15cm}{\begin{tabular}{| c | c | c  | c | }
	\hline
	 \makecell{adjacency \\ matrix} & \makecell{Randi\'c matrix \\ $R(\Phi)$}  & \makecell{Laplacian matrix \\ $L(\Phi)$} & \makecell{normalized Laplacian \\ matrix $\mathfrak{L}(\Phi)$ }  \\
	\hline
	$A(\Phi)$ & $R(\Phi)=D^{-1/2}A(\Phi)D^{-1/2}$  & $L(\Phi)=D-A(\Phi)$ & 
	\makecell{$\mathfrak{L}(\Phi)=D^{-1/2}L(\Phi)D^{-1/2}$ \\ $=I-R(\Phi)~~~$ } \\
	\hline
\end{tabular}}
\label{table4}
\end{center}
\end{table}

\section{Spectral properties of Hermitian Randi\'{c} matrix of second kind}
In this section, we characterize some spectral properties of $R^\omeg(X)$. We continue with some known results which are associated to our findings.

Let $\mathcal{M}_n(\mathbb{C})$ denote the set of all $n\times n$ matrices with complex entries. For $A \in \mathcal{M}_n(\mathbb{C})$, the matrix whose entries are absolute values of the corresponding entries of $A$ is denoted by $|A|$. The maximum of the absolute values of the eigenvalues of a matrix $A$ is called the \emph{spectral radius} of $A$. It is denoted by $\rho(A)$. Further, the spectrum of $A$ is denoted by $\spec (A)$.
\begin{theorem}[\cite{z}]
A $\T$-gain graph $(X,\varphi)$ is positive if and only if $(X,\varphi)\sim (X,\mathbf 1)$.
\end{theorem}
\begin{theorem}[\cite{mks}]
Let $(X,\varphi)$ be a connected and positive $\T$-gain graph. Then $X$ is bipartite if and only if $(X,-\varphi)$ is positive.
\end{theorem}
\begin{theorem}[\cite{hj}]
Let $A,B \in \mathcal{M}_n(\mathbb{C})$. Suppose $A$ is non-negative and irreducible, and $A\geq|B|$. Let $\lambda\coloneqq e^{i\theta}\rho(B)$ be a maximum-modulus eigenvalue of $B$. If $\rho(A)=\rho(B)$, then there is a diagonal unitary matrix $D\in \mathcal{M}_n(\mathbb{C})$ such that $B=e^{i\theta} DAD^{-1}$.
\end{theorem}

In \cite{kkp}, Kannan et al. studied the normalized Laplacian matrix for gain graphs. They also characterized some spectral properties for the Randi\'c matrix of an un-oriented graph.

\begin{lemma}[\cite{kkp}]
Let $X$ be a connected graph. Then $\spec(R(X))$= $\spec(-R(X))$ if and only if $X$ is bipartite.
\end{lemma}
\begin{lemma}[\cite{kkp}]
Let $\Phi_1$ and $\Phi_2$ be two connected gain graphs. If $\Phi_1\sim \Phi_2$, then
$$\spec (R(\Phi_1))=\spec (R(\Phi_2))$$
\end{lemma}
\begin{lemma}[\cite{kkp}]
If $\Phi\coloneqq(X,\varphi)$ is a connected gain graph, then
$$\rho(R(\Phi))\leq \rho(|R(\Phi)|)=\rho(R(X)).$$
\end{lemma}

The following result is an immediate consequence of the preceding lemmas.
\begin{theorem}
Let $X$ be a mixed graph. Then $\spec (R^\omeg(X)=\spec (R(X_U))$ if and only if $(X_U,\omeg)\sim (X_U, \mathbf{1})$.
\end{theorem}
\begin{proof}
If $\spec (R^\omeg(X))$=$\spec (R(X_U))$, then by Theorem 3.3 $R^\omeg(X)=e^{i\theta}D_\zeta R(X_U) D_\zeta^{-1}$, where $D_\zeta$ is a diagonal unitary matrix. Hence
\begin{align*}
&D_\zeta^{-1}R^\omeg(X)D_\zeta= e^{i\theta}R(X_U)\\
\text{or}, \;\;\;\;\; &D_\zeta^{-1}D^{-1/2}H^\omeg(X)D^{-1/2}D_\zeta= e^{i\theta}D^{-1/2}A(X_U)D^{-1/2}\\
\text{or}, \;\;\;\;\; & H^\omeg(X)=e^{i\theta}D_\zeta A(X_U)D_\zeta^{-1}.
\end{align*}
Since both the matrices $H^\omeg(X)$ and $A(X_U)$ are Hermitian, $\theta$ is either $0$ or $\pi$. This gives that either $(X_U,\omeg)\sim (X_U,\mathbf{1})$ or $(X_U,\omeg)\sim (X_U,-\mathbf{1})$. If $(X_U,\omeg)\sim (X_U,\mathbf{1})$, we are done. If $(X_U,\omeg)\sim (X_U,-\mathbf{1})$, then by Lemma 3.2, we have $\spec (R^\omeg(X))$=$\spec (-R(X_U))$. Again as $\spec (R^\omeg(X))$=$\spec (R(X_U))$, we have $\spec (R(X_U))$=$\spec (-R(X_U))$. Thus by Lemma 3.1, $X_U$ is bipartite. Now applying Theorem 3.2 for the positive gain graph $(X_U,-\omeg)$, we find that $(X_U,\omeg)$ is positive, and hence $(X_U,\omeg)\sim (X_U,\mathbf 1)$.

Conversely, if  $(X_U,\omeg)\sim (X_U,\mathbf 1)$, then clearly $\spec (R^\omeg(X))$=$\spec (R(X_U))$.
\end{proof}

In order to determine some spectral properties of the matrix $R^\omeg(X)$, we now provide the following lemma. In what follows, $h_{ij}$ always represents the $ij$-th entry of $H^\omeg(X)$.
\begin{lemma} \label{3.4}
 If $X$ be a mixed graph on $n$ vertices, and $\y\coloneqq(y_1,\ldots,y_n)^t\in\mathbb{C}^n$, then $$\y^*H^\omeg(X)\y=\sum\limits_{i<j,\; e_{ij}\in E(X_U)}(|y_i+h_{ij}y_j|^2-(|y_i|^2+|y_j|^2)).$$
\end{lemma}
\begin{proof}
We have
\begin{align*}
\y^*H^\omeg(X)\y=&\sum\limits_{i<j,\; e_{ij}\in E(X_U)}(\overline{y}_ih_{ij}y_j+y_i{\overline{h}_{ij}}\overline{y}_j)\\
=&\sum\limits_{i<j,\; e_{ij}\in E(X_U)}((\overline{y}_ih_{ij}y_j+y_i{\overline{h}_{ij}}\overline{y}_j)+\overline{y}_iy_i+\overline{y}_jy_j-\overline{y}_iy_i-\overline{y}_jy_j)\\
=&\sum\limits_{i<j,\; e_{ij}\in E(X_U)}((\overline{y}_ih_{ij}y_j+\overline{y}_iy_i)+(y_i{\overline{h}_{ij}}\overline{y}_j+\overline{y}_j\overline{h}_{ij}h_{ij}y_j)-(|y_i|^2+|y_j|^2))\\
=&\sum\limits_{i<j,\; e_{ij}\in E(X_U)}((y_i+h_{ij}y_j)(\overline{y_i+h_{ij}y_j})-(|y_i|^2+|y_j|^2)),\;\;\;\;\;\;\text{as}\; |h_{ij}|^2=1\\
=&\sum\limits_{i<j,\; e_{ij}\in E(X_U)}(|y_i+h_{ij}y_j|^2-(|y_i|^2+|y_j|^2)). \qedhere
\end{align*}
\end{proof}

Using the preceding lemma, we have the following theorem.
\begin{theorem}
Let $X$ be a mixed graph of order $n$, where $n\geq2$. If $\spec(R^\omeg(X))=\{\lambda_1,\ldots,\lambda_n\}$, then $ -1\leq\lambda_k\leq 1$ for each $k\in \{1,\ldots, n\}$.
\end{theorem}
\begin{proof}
For a complex vector $\x$, let $\y\coloneqq D^{-1/2}\x$, where $D$ is the diagonal degree matrix of the underlying graph $X_U$. Let $y_i$ be the $i$-th coordinate of $\y$. By Lemma~\ref{3.4}, we have 
\begin{align*}
\x^*R^\omeg(X)\x &=\y^*D^{1/2}D^{-1/2}H^\omeg(X)D^{-1/2}D^{1/2}\y\\
& =\y^*H^\omeg(X)\y\\
& =\sum\limits_{i<j,\; e_{ij}\in E(X_U)}(|y_i+h_{ij}y_j|^2-|y_i|^2-|y_j|^2).
\end{align*}
Also, $\x^*\x=\sum\limits^{n}_{i=1}d_i|y_i|^2$. By Courant-Fischer Theorem, we have
\begin{align*}
\lambda_k=&\underset{\x^{(1)},\ldots,\x^{(k-1)}\in \mathbb{C}^n}{\mathrm{max}}\;\;\; \underset{\underset{\x\in\mathbb{C}^n,\; \x\neq 0}{\x\perp \{\x^{(1)},\ldots,\x^{(k-1)}\}}}{\mathrm{min}}\frac{\x^*R^\omeg(X)\x}{\x^*\x}\\
=&\underset{\x^{(1)},\ldots,\x^{(k-1)}\in \mathbb{C}^n}{\mathrm{max}}\;\;\; \underset{\underset{x\in\mathbb{C}^n, \; \x\neq 0}{\x\perp \{\x^{(1)},\ldots,\x^{(k-1)}\}}}{\mathrm{min}}\frac{\sum\limits_{i<j,\; e_{ij}\in E(X_U)}(|y_i+h_{ij}y_j|^2-|y_i|^2-|y_j|^2)}{\sum\limits^n_{i=1}d_i|y_i|^2}.\numberthis\label{eq 1.3.1}
\end{align*}
Note that $|a+b|^2\leq 2|a|^2+2|b|^2 $ for two complex numbers $a$ and $b$. Therefore, we have

\begin{align*}
 \sum\limits_{i< j,\; e_{ij}\in E(X_U)}(|y_i+h_{ij}y_j|^2-|y_i|^2-|y_j|^2)\leq&\sum\limits_{i< j,\; e_{ij}\in E(X_U)}(2|y_i|^2+2|h_{ij}y_j|^2-|y_i|^2-|y_j|^2)\\
 =&\sum\limits_{i< j,\; e_{ij}\in E(X_U)}(|y_i|^2+|y_j|^2),\; \;\;\;\;\;\;\text{as}\;|h_{ij}|=1\\
 =&\sum\limits^n_{i=1}d_i|y_i|^2.\numberthis\label{eq 1.3.2}
 \end{align*}
 From Equation (\ref{eq 1.3.1}) and Equation (\ref{eq 1.3.2}), we have $\lambda_k\leq 1$. Again
 \begin{align*}
 \sum\limits_{i< j,\; e_{ij}\in E(X_U)}(|y_i+h_{ij}y_j|^2-|y_i|^2-|y_j|^2)\geq&\sum\limits_{i< j,\; e_{ij}\in E(X_U)}(-|y_i|^2-|y_j|^2)\\
=&-\sum\limits^n_{i=1}d_i|y_i|^2. \numberthis\label{eq 1.3.3}
\end{align*}
Hence from Equation (\ref{eq 1.3.1}) and Equation (\ref{eq 1.3.3}), we have $\lambda_k\geq -1.$
\end{proof}

Eigenvalue interlacing is a popular technique for generating inequality and regularity conclusions regarding graph structure in terms of eigenvalues. We provide an edge version of interlacing properties for $R^\omeg(X)$. First, we present two basic inequalities in the following lemma. The proofs are straight forward.
\begin{lemma}\label{l6}
Let $a$, $b$, and $c$ be three real numbers such that $b > 0$, $c > 0$ and $b-c > 0$.
\begin{enumerate}[(i)]
\item If $\frac{a}{b}\leq 1$, then $\frac{a-c}{b-c}\leq \frac{a}{b}$.
\item If $|\frac{a}{b}|\leq 1$, then $\frac{a+c}{b-c}\geq \frac{a}{b}$.
\end{enumerate}
\end{lemma}
\begin{theorem}
Let $X$ be a mixed graph on $n$ vertices and $X-e$ be the graph obtained by removing the edge $e$ of $X$. Let $\spec (R^\omeg(X))=\{\lambda_1,\ldots,\lambda_n\}$ and $\spec (R^\omeg(X-e))=\{\theta_1,\ldots,\theta_n\}$. Then $$\lambda_{k-1} \leq \theta_k \leq \lambda _{k+1}$$
for each $k\in \{1,\ldots,n\}$ with the convention that $\lambda_0=-1$ and $\lambda_{n+1}=1$.
\end{theorem}
\begin{proof}
From Equation (\ref{eq 1.3.1}), we have
\begin{align*}
\lambda_k=&\underset{\x^{(1)},\ldots,\x^{(k-1)}\in \mathbb{C}^n}{\mathrm{max}}\;\;\;\; \underset{\underset{\x\in\mathbb{C}^n, \; \x\neq 0}{\x\perp \{\x^{(1)},\ldots,\x^{(k-1)}\}}}{\mathrm{min}}\frac{\sum\limits_{i<j,\; e_{ij}\in E(X_U)}(|y_i+h_{ij}y_j|^2-|y_i|^2-|y_j|^2)}{\sum\limits^n_{i=1}d_i|y_i|^2},
\end{align*}
where $\y=D^{-1/2}\x$. Using min-max version of Courant-Fischer theorem, we can also write Equation (\ref{eq 1.3.1}) as
\begin{align}\label{minmax}
\lambda_k=&\underset{\x^{(k+1)},\ldots,\x^{(n)}\in \mathbb{C}^n}{\mathrm{min}}\;\;\;\; \underset{\underset{\x\in\mathbb{C}^n, \; \x\neq 0}{\x\perp \{\x^{(k+1)},\ldots,\x^{(n)}\}}}{\mathrm{max}}\frac{\sum\limits_{i<j,\; e_{ij}\in E(X_U)}(|y_i+h_{ij}y_j|^2-|y_i|^2-|y_j|^2)}{\sum\limits^n_{i=1}d_i|y_i|^2}.
\end{align}

Without loss of generality, let $e$ be the edge joining the vertex $v_1$ and $v_2$. After deleting the edge $e$, the degrees of the vertices $v_1$ and $v_2$ are decreased by 1. Hence for $G-e$, the expression $\sum\limits_{i<j,\; e_{ij}\in E(X_U)}(|y_i+h_{ij}y_j|^2-|y_i|^2-|y_j|^2)$ becomes $\sum\limits_{i< j,\; e_{ij}\in E(X_U)}(|y_i+h_{ij}y_j|^2-|y_i|^2-|y_j|^2)-(|y_1+h_{12}y_2|^2-|y_1|^2-|y_2|^2)$
and ${\sum\limits^n_{i=1}d_i|y_i|^2}\;\;\text{becomes}\;\;{\sum\limits^n_{i=1}d_i|y_i|^2}-|y_1|^2-|y_2|^2$. Therefore
\[\Resize{15.0cm}{
\theta_k= \underset{\x^{(1)},\ldots,\x^{(k-1)}\in \mathbb{C}^n}{\mathrm{max}}\;\;\; \underset{\underset{\x\in\mathbb{C}^n, \; \x\neq 0}{\x\perp \{\x^{(1)},\ldots,\x^{(k-1)}\}}}{\mathrm{min}}\frac{\scriptstyle{\sum\limits_{i< j,\; e_{ij}\in E(X_U)}((|y_i+h_{ij}y_j|^2-|y_i|^2-|y_j|^2)-(|y_1+h_{12}y_2|^2-|y_1|^2-|y_2|^2))}}{\sum\limits^n_{i=1}d_i|y_i|^2-|y_1|^2-|y_2|^2}.}\]

Let $x_i$ and $y_i$ be the $i$-th coordinates of $\x$ and $\y$, respectively. Choose $x_1$ and $x_2$ such that $\sqrt{d_2}x_1=h_{12}\sqrt{d_1}x_2$. Then$y_1=h_{12}y_2$, and so $|y_1+h_{12}y_2|^2=2(|y_1|^2+|y_2|^2)$. Thus for $a=|y_i+h_{ij}y_j|^2-|y_i|^2-|y_j|^2$, $b=\sum\limits^n_{i=1}d_i|y_i|^2$ and $c=2|y_1|^2$ in Lemma 3.5, we have 

$$\Resize{15.0cm}{\frac{\sum\limits_{i<j,\; e_{ij}\in E(X_U)}(|y_i+h_{ij}y_j|^2-|y_i|^2-|y_j|^2)-2|y_1|^2}{\sum\limits^n_{i=1}d_i|y_i|^2-2|y_1|^2}\leq \frac{\sum\limits_{i<j,\; e_{ij}\in E(X_U)}(|y_i+h_{ij}y_j|^2-|y_i|^2-|y_j|^2)}{\sum\limits^n_{i=1}d_i|y_i|^2}.}$$

 Note that if $\x\perp (\sqrt{d_2}e_1-h_{12}\sqrt{d_1}e_2)$, then $\sqrt{d_2}x_1=h_{12}\sqrt{d_1}x_2$, where the vectors $e_1, \;e_2$ are standard basis vectors of $\mathbb{C}^n$. Thus,
\begin{align*}
\theta_k\leq& \Resize{14.0cm}{\underset{\x^{(1)},\ldots,\x^{(k-1)}\in \mathbb{C}^n}{\mathrm{max}}\;\;\; \underset{\underset{\x\in\mathbb{C}^n,\; \x\neq 0, \; \sqrt{d_2}x_1=h_{12}\sqrt{d_1}x_2}{\x\perp \{\x^{(1)},\ldots,\x^{(k-1)}\}}}{\mathrm{min}}\frac{\sum\limits_{i<j,\; e_{ij}\in E(X_U)}(|y_i+h_{ij}y_j|^2-|y_i|^2-|y_j|^2)-2|y_1|^2}{\sum\limits^n_{i=1}d_i|y_i|^2-2|y_1|^2}}\\
=&\Resize{14.2cm}{\underset{\x^{(1)},\ldots,\x^{(k-1)}\in \mathbb{C}^n}{\mathrm{max}}\;\;\; \underset{\underset{\x\in\mathbb{C}^n, \; \x\neq 0}{\x\perp \{\x^{(1)},\ldots,\x^{(k-1)},\sqrt{d_2}e_1-h_{12}\sqrt{d_1}e_2\}}}{\mathrm{min}}\frac{\sum\limits_{i<j,\; e_{ij}\in E(X_U)}(|y_i+h_{ij}y_j|^2-|y_i|^2-|y_j|^2)-2|y_1|^2}{\sum\limits^n_{i=1}d_i|y_i|^2-2|y_1|^2}}\\
\leq& \Resize{14.2cm}{\underset{\x^{(1)},\ldots,\x^{(k-1)}\in \mathbb{C}^n}{\mathrm{max}}\;\;\; \underset{\underset{\x\in\mathbb{C}^n, \;\x\neq 0}{\x\perp \{\x^{(1)},\ldots,\x^{(k-1)},\sqrt{d_2}e_1-h_{12}\sqrt{d_1}e_2\}}}{\mathrm{min}}\frac{\sum\limits_{i<j,\; e_{ij}\in E(X_U)}(|y_i+h_{ij}y_j|^2-|y_i|^2-|y_j|^2)}{\sum\limits^n_{i=1}d_i|y_i|^2}}\\
\leq& \underset{\x^{(1)},\ldots,\x^k\in \mathbb{C}^n}{\mathrm{max}}\;\;\; \underset{\underset{\x\in\mathbb{C}^n, \;\x\neq 0}{\x\perp \{\x^{(1)},\ldots,\x^{(k)}\}}}{\mathrm{min}}\frac{\sum\limits_{i<j,\; e_{ij}\in E(X_U)}(|y_i+h_{ij}y_j|^2-|y_i|^2-|y_j|^2)}{\sum\limits^n_{i=1}d_i|y_i|^2}=\lambda_{k+1}.
\end{align*}

Similarly, from Equation (\ref{minmax}) we have

\begin{align*}
\theta_k =&\Resize{13.9cm}{\underset{\x^{(k+1)},\ldots,\x^{(n)}\in \mathbb{C}^n}{\mathrm{min}}\;\;\; \underset{\underset{\x\in\mathbb{C}^n, \; \x\neq 0}{\x\perp \{\x^{(k+1)},\ldots,\x^{(n)}\}}}{\mathrm{max}}\frac{\sum\limits_{i<j,\; e_{ij}\in E(X_U)}(|y_i+h_{ij}y_j|^2-|y_i|^2-|y_j|^2)-(|y_1+h_{12}y_2|^2-|y_1|^2-|y_2|^2)}{\sum\limits^n_{i=1}d_i|y_i|^2-|y_1|^2-|y_2|^2}}\\
& \geq \Resize{13.5cm}{\underset{\x^{(k+1)},\ldots,\x^{(n)}\in \mathbb{C}^n}{\mathrm{min}}\;\;\; \underset{\underset{\x\in\mathbb{C}^n, \; \x\neq 0, \sqrt{d_2}x_1=-h_{12}\sqrt{d_1} x_2}{\x\perp \{\x^{(k+1)},\ldots,\x^{(n)}\}}}{\mathrm{max}}\frac{\scriptstyle{\sum\limits_{i<j,\; e_{ij}\in E(X_U)}(|y_i+h_{ij}y_j|^2-|y_i|^2-|y_j|^2)-(|y_1+h_{12}y_2|^2-|y_1|^2-|y_2|^2)}}{\sum\limits^n_{i=1}d_i|y_i|^2-|y_1|^2-|y_2|^2}}\\
&=\Resize{13.5cm}{ \underset{\x^{(k+1)},\ldots,\x^{(n)}\in \mathbb{C}^n}{\mathrm{min}}\;\;\; \underset{\underset{\x\in\mathbb{C}^n, \;\x\neq 0}{\x\perp \{\x^{(k+1)},\ldots,\x^{(n)},\sqrt{d_2}e_1+h_{12} \sqrt{d_1} e_2\}}}{\mathrm{max}}\frac{\sum\limits_{i<j,\; e_{ij}\in E(X_U)}(|y_i+h_{ij}y_j|^2-|y_i|^2-|y_j|^2)+2|y_1|^2}{\sum\limits^n_{i=1}d_i|y_i|^2-2|y_1|^2}.}
\end{align*}

Again taking $a=|y_i+h_{ij}y_j|^2-|y_i|^2-|y_j|^2$, $b=\sum\limits^n_{i=1}d_i|y_i|^2$ and $c=2|y_1|^2$, we find from Equations (\ref{eq 1.3.2}) and (\ref{eq 1.3.3}) that $|\frac{a}{b}|\leq 1$. Therefore Lemma \ref{l6} (ii) gives 

$$\Resize{15.0cm}{\frac{\sum\limits_{i<j,\; e_{ij}\in E(X_U)}(|y_i+h_{ij}y_j|^2-|y_i|^2-|y_j|^2)+2|y_1|^2}{\sum\limits^n_{i=1}d_i|y_i|^2-2|y_1|^2}\geq \frac{\sum\limits_{i<j,\; e_{ij}\in E(X_U)}(|y_i+h_{ij}y_j|^2-|y_i|^2-|y_j|^2)}{\sum\limits^n_{i=1}d_i|y_i|^2}.}$$ 
Thus
\begin{align*}
\theta_k & \geq \Resize{14.0cm}{\underset{\x^{(k+1)},\ldots,\x^{(n)}\in \mathbb{C}^n}{\mathrm{min}}\;\;\; \underset{\underset{\x\in\mathbb{C}^n, \;\x\neq 0}{\x\perp \{\x^{(k+1)},\ldots,\x^{(n)},\sqrt{d_2}e_1+h_{12} \sqrt{d_1} e_2\}}}{\mathrm{max}}\frac{\sum\limits_{i<j,\; e_{ij}\in E(X_U)}(|y_i+h_{ij}y_j|^2-|y_i|^2-|y_j|^2)}{\sum\limits^n_{i=1}d_i|y_i|^2}}\\
& \geq \underset{\x^{(k)},\ldots,\x^{(n)}\in \mathbb{C}^n}{\mathrm{min}}\;\;\; \underset{\underset{\x\in\mathbb{C}^n, \;\x\neq 0}{\x\perp \{\x^{(k)},\ldots,\x^{(n)}\}}}{\mathrm{max}}\frac{\sum\limits_{i<j,\; e_{ij}\in E(X_U)}(|y_i+h_{ij}y_j|^2-|y_i|^2-|y_j|^2)}{\sum\limits^n_{i=1}d_i|y_i|^2}\\
&=\lambda_{k-1}.
\end{align*}

Thus, $\lambda_{k-1} \leq \theta_k \leq \lambda _{k+1}$ with the convention that $\lambda_0=-1$ and $\lambda_{n+1}=1$.
\end{proof}

Let $S_H(X)\coloneqq(s_{H_{ke}})$ be an $n\times m$ matrix indexed by the vertices and edges of a mixed graph $X$, with $|s_{H_{ke}}|=1$ whenever $k$ is incident to $e$, and
$$
s_{H_{ke}}=\left\{
  \begin{array}{ll}
    -s_{H_{\ell e}} & \hbox{if $e=e_{k\ell}$} \\
    -\omega s_{H_{\ell e}} & \hbox{if $e=\overrightarrow{e_{k\ell}}$} \\
    \;\; \;0 & \hbox{otherwise.}
  \end{array}
\right.
$$

If $D$ is the diagonal degree matrix of the underlying graph $X_U$, $D^{-1/2}S_H(X)\coloneqq (s_{ke})$ and $\left(D^{-1/2}S_H(X)\right)\left(D^{-1/2}S_H(X)\right)^*=(\alpha_{k\ell})_{n\times n}$, then
$$\alpha_{k\ell}=\sum\limits_{e\in E(X)}s_{ke}\overline s_{\ell e}=\sum\limits_{e\in E(X)}\frac{1}{\sqrt{d_kd_\ell}}s_{H_{ke}}\overline s_{H_{\ell e}}.$$
Thus $\alpha_{kk}=\sum\limits_{e\in E(X)}\frac{1}{\sqrt{d_kd_k}}s_{H_{ke}}\overline s_{H_{ke}}=\sum\limits_{e\in E(X)}\frac{1}{d_k}|s_{H_{ke}}|^2=\frac{1}{d_k}.  d_k$=1. Now assume that $k\neq \ell$.
\begin{enumerate}
\item[(i)] For ${e_{k\ell}}\in E(X)$,
$$\Resize{14.0cm}{\alpha_{k\ell}=s_{ke}\overline s_{le}=\frac{1}{\sqrt{d_k}}s_{H_{ke}}\frac{1}{\sqrt{d_\ell}}\overline s_{H_{\ell e}}=\frac{1}{\sqrt{d_kd_\ell}}(-s_{H_{\ell e}})\overline s_{H_{le}}=-\frac{1}{\sqrt{d_kd_\ell}}\left|s_{H_{\ell e}}\right|^2=-\frac{1}{\sqrt{d_kd_\ell}}.}$$
\item[(ii)] For $\overrightarrow{e_{k\ell}}\in E(X)$,
$$\alpha_{k\ell}=s_{ke}\overline s_{\ell e}=\frac{1}{\sqrt{d_k}}s_{H_{ke}}\frac{1}{\sqrt{d_\ell}}\overline s_{H_{\ell e}}=\frac{1}{\sqrt{d_kd_\ell}}(-\omega s_{H_{le}})\overline s_{H_{\ell e}}=\frac{-\omega}{\sqrt{d_kd_\ell}}\left|s_{H_{\ell e}}\right|^2=\frac{-\omega}{\sqrt{d_kd_\ell}}.$$
\item[(iii)] For $\overrightarrow{e_{\ell k}}\in E(X)$,
$$\alpha_{k\ell}=s_{ke}\overline s_{\ell e}=\frac{1}{\sqrt{d_k}}s_{H_{ke}}\frac{1}{\sqrt{d_\ell}}\overline s_{H_{\ell e}}=\frac{1}{\sqrt{d_kd_\ell}}(s_{H_{k e}})\overline{-\omega s_{H_{k e}}}=\frac{-\overline{\omega}}{\sqrt{d_kd_\ell}}\left|s_{H_{k e}}\right|^2=\frac{-\overline{\omega}}{\sqrt{d_kd_\ell}}.$$
\end{enumerate}

Thus, $R^\omeg(X)=I-\left(D^{-1/2}S_H(X)\right)\left(D^{-1/2}S_H(X)\right)^*$.
\begin{lemma}[\cite{yq}]\label{lem 3.5}
 A mixed graph $X$ is positive if and only if for any two vertices $v_i$ and $v_j$ all paths from $v_i$ to $v_j$ have the same value.
 \end{lemma}
 \begin{theorem}
 Let $X$ be a connected mixed graph. If 1 is an eigenvalue of $R^\omeg(X)$, then $X$ is positive and 1 is a simple eigenvalue of $R^\omeg(X)$.
 \end{theorem}
 \begin{proof}
 Assume that $1$ is an eigenvalue of $R^\omeg(X)$ with corresponding eigenvector $\x$. If $\x=(x_1,\ldots,x_n)^t$, we have
 \begin{align*} &R^\omeg(X)\x=\x  \\
 \text{or,} \;\;\;\;\; &(I-R^\omeg(X))\x=0\\
 \text{or,} \;\;\;\;\; &(D^{-1/2}S_H(X)(D^{-1/2}S_H(X))^*)\x=0\\
 \text{or,} \;\;\;\;\; &\langle D^{-1/2}S_H(X)(D^{-1/2}S_H(X))^*\x,\x \rangle=0\\
 \text{or,} \;\;\;\;\; &\langle (D^{-1/2}S_H(X))^*\x,(D^{-1/2}S_H(X))^*\x \rangle=0\\
 \text{or,}\;\;\;\;\; & (D^{-1/2}S_H(X))^*\x=0.
 \end{align*}
Thus if $e$ is an edge of $X$ with end vertices $v_i$ and $v_j$, we have $((D^{-1/2}S_H(X))^*\x)_e=0$, and this gives $\overline{s}_{H_{ie}}d_i^{-1/2}x_i+\overline{s}_{H_{je}}d_j^{-1/2}x_j=0$.


Note that $s_{H_{ie}}\overline{s}_{H_{je}}=-h_{ij}$, and so $x_i=\sqrt{\frac{d_i}{d_j}}h_{ij}x_j$ for any edge incident to $v_i$ and $v_j$. Let $W_{1k}\coloneqq u_1u_2\ldots u_k$ be any $u_1u_k$-path such that $u_1=v_1$ and  $u_k=v_j$. Also, let $W_{1r}$ be the $u_1u_r$-section of the path $W_{1k}$, where $2\leq r\leq k$. For $W_{1r}=u_1u_2\ldots u_{r-1}u_r$, let $h(W_{1r})=h_{12}\ldots h_{(r-1)r}$, the value of $W_{1r}$.

We have $x_1=\sqrt{\frac{d_1}{d_2}}h_{12}x_2=\sqrt{\frac{d_1}{d_3}}h_{12}h_{23}x_3=\cdots=\sqrt{\frac{d_1}{d_i}}h(W_{1i})x_i.$ This implies that each $v_iv_j$-path has the same value. Hence by Lemma~\ref{lem 3.5}, $X$ is positive. Moreover, $\x=(x_1,\ldots, x_i)^t =(x_1,\sqrt{\frac{d_2}{d_1}}\overline{h(W_{12})}x_1,\ldots,\sqrt{\frac{d_i}{d_1}}\overline{h(W_{1i})}x_1)^t$, so $$\x =x_1\bigg(1,\sqrt{\frac{d_2}{d_1}}\overline{h(W_{12})},\ldots,\sqrt{\frac{d_i}{d_1}}\overline{h(W_{1i})}\bigg)^t.$$

Hence $1$ is an eigenvalue of $R^\omeg(X)$ with multiplicity $1$.
 \end{proof}
Yu et al. \cite{ydsj}, in their study of Hermitian normalized Laplacian matrix for mixed networks, established that a graph is bipartite if and only if all of its eigenvalues are symmetric about 1. The symmetric characteristics of the $R^\omeg(X)$ eigenvalues can also be determined in a similar manner.
\begin{theorem}
If $X$ is a connected mixed graph, then $X$ is bipartite if and only if all eigenvalues of $R^\omeg(X)$ are symmetric about 0.
\end{theorem}
\begin{proof}
Because of $R^\omeg(X)=I-\mathfrak{L}^\omeg(X)$, the proof is analogous to the proof of Theorem 3.5 in \cite{ydsj}.
\end{proof}
\begin{theorem}[\cite{ydsj}] \label{th 1.3.1}
If $X$ is a connected mixed graph, then $2$ is an eigenvalue of $\mathfrak{L}^\omeg(X)$ if and only if $X$ is a positive bipartite graph.
\end{theorem}
Noting that $R^\omeg(X)=I-\mathfrak{L}^\omeg(X)$, we get the following corollary from Theorem \ref{th 1.3.1}.
\begin{corollary}
If $X$ is a connected mixed graph, then $-1$ is an eigenvalue of $R^\omeg(X)$ if and only if $X$ is a positive bipartite graph.
\end{corollary}
Note that if $X$ is a bipartite mixed graph, then the spectrum of $R^\omeg(X)$ is symmetric about $0$. As a result, if $X$ is a bipartite mixed connected graph, then $1$ is an eigenvalue of $R^\omeg(X)$ if and only if $X$ is a positive.
\section{Determinant and Characteristic Polynomial of $R^\omeg(X)$}
In this section, we provide some results similar to Theorem 2.7 in \cite{ll} and Proposition 7.3 in \cite{b2} for Hermitian Randi\'c matrix of second kind. Lu et al. in \cite{lwz} defined the Hermitian-Randi\'c matrix $R_H(X)$ of a mixed graph $X$. For this Hermitian matrix, they obtained the determinant and characteristic polynomial. In \cite{lwz}, $\de R_H(X)$ is expressed as a summation, in which the summation is taken over a specific class of elementary spanning sub-graphs of $X$. In the next theorem, we find an analogous expression for $\de R^\omeg(X)$, in which the summation is taken over all spanning elementary sub-graphs of $X$. 


%

Let $X'$ be an elementary sub-graph of a mixed graph $X$ on $n$ vertices. Let $c(X')$ be the number of components of $X'$, and $r(X')=n-c(X')$. Further, let $s(X')$ be the number of cycles of length at least 3 in $X'$. For a sub-graph $Y$ of $X$, let $Q(Y)=\prod_{v_i\in V(Y)}\frac{1}{d_i}$.

Recall that a cycle is called positive or negative according as its value is positive or negative, respectively. A cycle $C$ is called \emph{semi-positive} if its value is either $\omega Q(C)$ or $\overline{\omega} Q(C)$. Similarly, it is called \emph{semi-negative} if its value is either $-\omega Q(C)$ or $-\overline{\omega} Q(C)$. Let $l_p(X'),l_n(X'), l_{sp}(X')$ and $l_{sn}(X')$ be the number of positive, negative, semi-positive and  semi-negative cycles in $X'$, respectively.

\begin{theorem}\label{th 1.4.0}
Let $R^\omeg(X)$ be the Hermitian Randi\'c matrix of second kind of a mixed graph $X$ of order $n$. Then
$$\de (R^\omeg(X))=\sum\limits_{X'}(-1)^{r(X')+l_n(X')+l_{sn}(X')}2^{l_n(X')+l_p(X')}Q(X'),$$
where the summation is over all spanning elementary sub-graphs $X'$ of $X$.
\end{theorem}
\begin{proof}
Let $X$ be a mixed graph of order $n$. We have
$$\de (R^\omeg(X))=\sum\limits_{\pi\in S_n}\sgn (\pi)R^\omeg_{1\pi(1)}R^\omeg_{2\pi(2)}\cdots R^\omeg_{n\pi(n)},$$
where $S_n$ is the set of all permutations on $\{1,\ldots,n\}$.

Consider a term $\sgn(\pi)R^\omeg_{1\pi(1)}\cdots R^\omeg_{n\pi(n)}$ in the expansion of $\de (R^\omeg(X))$. If $v_kv_{\pi(k)}$ is not an edge of $X$, then $R^\omeg_{k\pi(k)}=0$, hence the term vanishes. Thus, if the term corresponding to a permutation is non-zero, then it is fixed- point-free and can be expressed uniquely as the composition of disjoint cycles of length at least $2$. Consequently, each non-vanishing term in the expansion of $\de (R^\omeg(X))$ gives rise to a spanning elementary sub-graph $X'$ of $X$. Note that a spanning elementary sub-graph may correspond to several non-vanishing terms in the expansion of $\de (R^\omeg(X))$.

Let $X'$ be a spanning elementary sub-graph of $X$ that corresponds to a non-vanishing term in the expansion of $\de (R^\omeg(X))$. Let $\pi(X')$ be the set of all permutations that correspond to $X'$. Clearly, $|\pi(X')|=2^{s(X')}$, and $\sgn(\pi)=(-1)^{r(X')}$ for $\pi\in \pi(X')$. Thus
$$\de R^\omeg(X)=\sum_{X'} (-1)^{r(X')}\sum_{\pi\in \pi(X')}R^\omeg_{1\pi(1)}R^\omeg_{2\pi(2)}\cdots R^\omeg_{n\pi(n)}.$$ 

Note that, for each edge component with vertices $v_k$ and $v_\ell$, the corresponding factor $R^\omeg_{k\ell}R^\omeg_{\ell k}$ has the value $\frac{1}{\sqrt{d_kd_\ell}}\cdot\frac{1}{\sqrt{d_\ell d_k}}=\frac{1}{{d_kd_\ell}}$ or $\frac{\omega}{\sqrt{d_kd_\ell}}\cdot\frac{\overline{\omega}}{\sqrt{d_\ell d_k}}=\frac{1}{{d_kd_\ell}}$. Furthermore, if for one direction the value of a mixed cycle is $\alpha$, then for the reversed direction its value is $\overline{\alpha}$, the conjugate of $\alpha$. Thus, in the summation of $\de (R^{\omeg}(X))$, we have two cases for the cycles having complex values. For a semi-positive cycle, say $C_1$, we have $\omega\prod_{v_j\in V(C_1)}\frac{1}{d_j}+\overline{\omega}\prod_{j\in V(C_1)}\frac{1}{d_j}=(\omega+\overline{\omega})\prod_{v_j\in V(C_1)}\frac{1}{d_j}=\prod_{v_j\in V(C_1)}\frac{1}{d_j}$, and for a semi-negative cycle, say $C_2$, we have $-\omega\prod_{v_j\in V(C_2)}\frac{1}{d_j}-\overline{\omega}\prod_{j\in V(C_2)}\frac{1}{d_j}=-(\omega+\overline{\omega})\prod_{v_j\in V(C_2)}\frac{1}{d_j}=-\prod_{v_j\in V(C_2)}\frac{1}{d_j}$. In addition, if a cycle $C$ has the real values $\prod_{v_j\in V(C)}\frac{1}{d_j}$ or $-\prod_{v_j\in V(C)}\frac{1}{d_j}$ for some direction, then it has the same value for the other direction of $C$ as well.

Let $X'_e$ be the sub-graph of $X'$ consisting of the edge components of $X'$. Recall that $Q(X'_e)=\prod_{v_i\in V(X'_e)} \frac{1}{d_i}$. As a convention, assume that $Q(X'_e)=1$ if $X'_e=\phi$. Let $C_1,\;C_2,\ldots, C_{s(X')}$ be the cycles of $X'$ of a length at least 3, and for $i\in \{1,\ldots,s(X')\}$ and $\alpha \in \{1,2\}$, define
$$
W^\alpha_i=\left\{
  \begin{array}{ll}
    W(C_i) & \text{ if   } \alpha=1 \\
    \overline{W(C_i)} & \text{ if  } \alpha=2.
  \end{array}
\right.
$$
Thus 
$$\de R^\omeg(X)=\sum_{X'} (-1)^{r(X')}\sum_{\alpha \in \{1,2\}}Q(X'_e)W^{\alpha}_1 W^{\alpha}_2\ldots W^{\alpha}_{s(X')} .$$
Observe that for $1\leq k\leq s(X')$, we have 
\begin{align*}
&W^{\alpha}_{1}\ldots W^{\alpha}_{k-1}W^{1}_{k}W^{\alpha}_{k+1}\ldots W^{\alpha}_{s(X')}+W^{\alpha}_{1}\ldots W^{\alpha}_{k-1}W^{2}_{k}W^{\alpha}_{k+1}\ldots W^{\alpha}_{s(X')}\\
=& \left\{
  \begin{array}{ll}
    (-1)^{\beta_k} W^{\alpha}_{1}\ldots W^{\alpha}_{k-1}Q(C_k)W^{\alpha}_{k+1}\ldots W^{\alpha}_{s(X')} & \text{ if $W(C_k)$ is not real  }\\
    (-1)^{\beta_k} 2W^{\alpha}_{1}\ldots W^{\alpha}_{k-1}Q(C_k)W^{\alpha}_{k+1}\ldots W^{\alpha}_{s(X')} & \text{ if $W(C_k)$ is real },
  \end{array}
\right.
\end{align*}
where 
$$
\beta_k=\left\{
  \begin{array}{ll}
    1 & \text{ if $C_k$ is negative or semi-negative cycle} \\
    0 & \text{ otherwise.}
  \end{array}
\right.
$$
Thus
\begin{align*}
&\sum_{\alpha \in \{1,2\}}Q(X'_e)W^{\alpha}_1 W^{\alpha}_2\ldots W^{\alpha}_{s(X')}\\
=&(-1)^{l_n(X')+l_{sn}(X')}2^{s(X')-l_{sn}(X')-l_{sp}(X')}Q(X'_e)Q(C_1)\ldots Q(C_{s(X')})\\
=&(-1)^{l_n(X')+l_{sn}(X')}2^{l_{n}(X')+l_{p}(X')}Q(X').\;\; \qedhere
\end{align*} 
\end{proof}

Let $P_{R^\omeg}(X,x)\colon = \de(xI-R^\omeg(X))$ be the characteristic polynomial of the matrix $R^\omeg(X)$ of a mixed graph $X$. Now we compute an expression for the coefficients of $P_{R^\omeg}(X,x)$.
\begin{theorem}
 If $P_{R^\omeg}(X,x)\coloneqq x^n+a_1x^{n-1}+\cdots+a_n$, then
$$(-1)^ka_k=\sum\limits_{X'}(-1)^{r(X')+l_n(X')+l_{sn}(X')}2^{l_n(X')+l_p(X')}Q(X'),$$
where the summation is over all elementary sub-graphs $X'$ with order $k$ of $X$.
\end{theorem}
\begin{proof}
The proof is based on Theorem \ref{th 1.4.0}, and makes use of the fact that the summation of the determinants of all principal $k\times k$ sub-matrices of $R^\omeg(X)$ is $(-1)^ka_k$.
\end{proof}
In the next corollary, we look at how the coefficients of $P_{R^\omeg}(X,x)$ change their shape for different graph structures.
\begin{corollary}\label{c 1.4.1}
Let $P_{R^\omeg}(X,x)\coloneqq x^n+a_1x^{n-1}+\cdots+a_n$.
\begin{enumerate}[(i)]
\item If $X$ is a tree, then $(-1)^ka_k=\sum\limits_{X'}(-1)^{r(X')}Q(X')$, where the summation is over all elementary sub-graphs $X'$ with order $k$ of $X$.
\item If the underlying graph $X_U$ of $X$ is $\delta$-regular ($\delta\neq 0$), then
    $$(-1)^ka_k=\sum\limits_{X'}(-1)^{r(X')+l_n(X')+l_{sn}(X')}2^{l_n(X')+l_p(X')}\frac{1}{\delta^k},$$
where the summation is over all elementary sub-graphs $X'$ with order $k$ of $X$ .
\end{enumerate}
\end{corollary}
The proof of Corollary \ref{c 1.4.1} is straightforward due to the absence of cycles in a tree and the fact that every vertex in a $\delta$-regular graph has degree $\delta$.

%
%
In 1999, Bollob\'as et al. \cite{bol} defined the general Randi\'c index $R^{(\alpha)}(X)$ of an un-oriented graph $X$ as $$R^{(\alpha)}(X)=\sum_{u\sim v}(d_u d_v)^\alpha.$$
Now we find a bound for eigenvalues of $R^\omeg(X)$ in terms of $R^{(-1)}(X_U)$.
\begin{theorem}
If $\lambda_1$ is the smallest eigenvalue of $R^\omeg(X)$, then
$$\lambda^2_1\geq \frac{2R^{(-1)}(X_U)}{n(n-1)}.$$
\end{theorem}
\begin{proof}
Let the eigenvalues $\lambda_1,\ldots,\lambda_n$ of $R^\omeg(X)$ satisfy $\lambda_1\leq\cdots\leq\lambda_k\leq\lambda_{k+1}\leq\cdots\leq\lambda_n$. We have
\begin{align*}
\sum_{i=1}^{n}{\lambda_i}^2=\text{trace}(R^\omeg(X)^2)=\sum_{i=1}^{n}\sum_{j=1}^{n}R^\omeg_{ij}R^\omeg_{ji}=&\sum_{i=1}^{n}\sum_{j=1}^{n}R^\omeg_{ij}{\overline R^\omeg_{ij}}\\
=&\sum_{i=1}^{n}\sum_{j=1}^{n}|R^\omeg_{ij}|^2\\
=&2\sum_{i\thicksim j}\frac{1}{d_id_j} = 2R^{(-1)}(X_U).
\end{align*}

Also,
\begin{align*}
&\sum_{i=1}^{n}(\lambda_i-\lambda_1)=\sum_{i=1}^{n}{\lambda_i-n\lambda_1}=-n\lambda_1\\
\text{or,}  \;\;\;\;\;\;\;\;\;\ &\sum_{i=1}^{n}(\lambda_i-\lambda_1)^2+\sum_{p, q=1,\; p\neq q}^{n}(\lambda_p-\lambda_1)(\lambda_q-\lambda_1)=(n\lambda_1)^2.
\end{align*}
Since $\sum_{p\neq q}(\lambda_p-\lambda_1)(\lambda_q-\lambda_1)$ is non-negative, we have	
\begin{align*}
&\sum_{i=1}^{n}(\lambda_i-\lambda_1)^2\leq n^2\lambda^2_1\\
\text{or,} \;\;\;\;\;\;\;\;\;\ &\sum_{i=1}^{n}{\lambda^2_i}+n\lambda^2_1\leq n^2\lambda^2_1\\
\text{or,} \;\;\;\;\;\;\;\;\;\ &2R^{(-1)}(X_U)+n\lambda^2_1 \leq n^2\lambda^2_1\\
\text{or,} \;\;\;\;\;\;\;\;\;\ &\lambda^2_1\geq \frac{2R^{(-1)}(X_U)}{n(n-1)}. \qedhere
\end{align*}
\end{proof}
For an $n\times n$ matrix $A\coloneqq (a_{ij})$, define
\begin{align*}
&\gamma_1(A)=\mi\bigg\{\frac{1}{n}\sum\limits_{i=1}^{n}\sum\limits_{j=1}^{n}a_{ij}\;,\;\frac{1}{n}\sum\limits_{i=1}^{n}a_{ii}- \frac{1}{n(n-1)}\sum\limits_{i\neq j}a_{ij}\bigg\} \;\;\;\;\;\text{and}\\
&\gamma_2(A)=\ma\bigg\{\frac{1}{n}\sum\limits_{i=1}^{n}\sum\limits_{j=1}^{n}a_{ij}\;,\;\frac{1}{n}\sum\limits_{i=1}^{n}a_{ii}- \frac{1}{n(n-1)}\sum\limits_{i\neq j}a_{ij}\bigg\}. 
\end{align*}
\begin{lemma}[\cite{m}]\label{l 1.4.1}
Let $A\coloneqq (a_{ij})$ be an $n\times n$ Hermitian matrix. Let $\lambda_1$ and $\lambda_n$ be the smallest and largest eigenvalues of $A$, respectively. Then $\lambda_1\leq\gamma_1(A)\leq\gamma_2(A)\leq\lambda_n$.
\end{lemma}
\begin{theorem}\label{th 1.4.2}
Let $X$ be a mixed graph, and let $\lambda_1$ and $\lambda_n$ be the smallest and the largest eigenvalues of $R^\omeg(X)$, respectively. Then
 $$\lambda_1\leq-\frac{1}{n(n-1)}\left(\sum\limits_{i\sim j}\frac{2}{\sqrt{d_id_j}}+\sum\limits_{i\rightarrow j}\frac{1}{\sqrt{d_id_j}}\right)\leq \frac{1}{n}\left(\sum\limits_{i\sim j}\frac{2}{\sqrt{d_id_j}}+\sum\limits_{i\rightarrow j}\frac{1}{\sqrt{d_id_j}}\right)\leq \lambda_n.$$
\end{theorem}
\begin{proof}
We have
\begin{align*}
\sum\limits_{i\neq j}R_{ij}^\omeg=\sum\limits_{i=1}^{n}\sum\limits_{j=1}^{n}R_{ij}^\omeg=&\sum\limits_{i\sim j}\frac{2}{\sqrt{d_id_j}}+\sum\limits_{i\rightarrow j}\bigg(\frac{\omega}{\sqrt{d_id_j}}+\frac{\overline{\omega}}{\sqrt{d_id_j}}\bigg)\\
=& \sum\limits_{i\sim j}\frac{2}{\sqrt{d_id_j}}+\sum\limits_{i\rightarrow j}\frac{1}{\sqrt{d_id_j}}.
\end{align*}

Also,
$$\sum\limits_{i=1}^{n}R_{ii}^\omeg=\text{trace}(R^\omeg)=0.$$

Hence by Lemma \ref{l 1.4.1}, we have
$$\lambda_1\leq\gamma_1\leq\gamma_2\leq\lambda_n, \;\; \text{where}$$
\begin{align*}
\gamma_1 &=\mi\bigg\{\frac{1}{n}\bigg(\sum\limits_{i\sim j}\frac{2}{\sqrt{d_id_j}}+\sum\limits_{i\rightarrow j}\frac{1}{\sqrt{d_id_j}}\bigg)\;,\; 0-\frac{1}{n(n-1)}\bigg(\sum\limits_{i\sim j}\frac{2}{\sqrt{d_id_j}}+\sum\limits_{i\rightarrow j}\frac{1}{\sqrt{d_id_j}}\bigg)\bigg\}\\
&=-\frac{1}{n(n-1)}\bigg(\sum\limits_{i\sim j}\frac{2}{\sqrt{d_id_j}}+\sum\limits_{i\rightarrow j}\frac{1}{\sqrt{d_id_j}}\bigg),\; \text{ and}\\
\gamma_2 &=\ma\bigg\{\frac{1}{n}\bigg(\sum\limits_{i\sim j}\frac{2}{\sqrt{d_id_j}}+\sum\limits_{i\rightarrow j}\frac{1}{\sqrt{d_id_j}}\bigg)\;,\; 0-\frac{1}{n(n-1)}\bigg(\sum\limits_{i\sim j}\frac{2}{\sqrt{d_id_j}}+\sum\limits_{i\rightarrow j}\frac{1}{\sqrt{d_id_j}}\bigg)\bigg\}\\
&=\frac{1}{n}\bigg(\sum\limits_{i\sim j}\frac{2}{\sqrt{d_id_j}}+\sum\limits_{i\rightarrow j}\frac{1}{\sqrt{d_id_j}}\bigg).
\end{align*}
Hence
\[\lambda_1\leq\frac{-1}{n(n-1)}\bigg(\sum\limits_{i\sim j}\frac{2}{\sqrt{d_id_j}}+\sum\limits_{i\rightarrow j}\frac{1}{\sqrt{d_id_j}}\bigg)\leq\frac{1}{n}\bigg(\sum\limits_{i\sim j}\frac{2}{\sqrt{d_id_j}}+\sum\limits_{i\rightarrow j}\frac{1}{\sqrt{d_id_j}}\bigg)\leq\lambda_n.  \qedhere \]
\end{proof}

The following corollary follows easily from Theorem \ref{th 1.4.2}.
\begin{corollary}
Let $X$ be a mixed graph. If $\lambda_1$ and $\lambda_n$ are the smallest and the largest eigenvalues of $R^\omeg(X)$, then
 $$\lambda_n-\lambda_1\geq \frac{1}{n-1}\bigg(\sum\limits_{i\sim j}\frac{2}{\sqrt{d_id_j}}+\sum\limits_{i\rightarrow j}\frac{1}{\sqrt{d_id_j}}\bigg).$$
\end{corollary}
\section{Energy of $R^{\omeg}(X)$}

Lu et al. in \cite{lwz} investigated the energy for Hermitian Randi\'c matrix $R_H(X)$ and computed various bounds. We analogously define the energy $\varepsilon(R^\omeg(X))$ of $R^\omeg(X)$. That is, $\varepsilon(R^\omeg(X))$ is the sum of the absolute values of the eigenvalues of $R^\omeg(X)$.  We find that most of the results on energy of $R_H(X)$ also hold good for the matrix $R^\omeg(X)$ due to the fact that $\trace(R^\omeg(X))=\trace(R_H(X))=0$ and $\sum_{i=1}^{n}{\lambda_i}^2=2R^{(-1)}(X_U)$. 

\begin{theorem}
Let $X$ be a mixed graph of order $n$, and $\lambda_1,\ldots,\lambda_n$ be the eigenvalues of $R^{\omeg}(X)$. Then
$$\sqrt{2R^{(-1)}(X_U)+n(n-1)(\de R^{\omeg}(X))^{2/n}}\leq \varepsilon(R^{\omeg}(X)) \leq \sqrt{2nR^{(-1)}(X_U)},$$
where equality holds if $|\lambda_1|=\cdots=|\lambda_n|$.
\end{theorem}
\begin{proof}
The proof is similar to the proof of Theorem 3.5 in \cite{lwz}, that can be obtained using the Cauchy-Schwartz inequality and geometric-arithmetic inequality.
\end{proof}
\begin{theorem}
Let $X$ be a mixed graph, and $\lambda_1,\ldots,\lambda_n$  be the eigenvalues of $R^{\omeg}(X)$, where $\lambda_1\leq\ldots\leq\lambda_2\leq\ldots\leq\lambda_n$, and $k$ is the number of negative eigenvalues. Then $$\varepsilon(R^{\omeg}(X))\geq 2(n-k) \bigg(\frac{\de(R^\omeg(X))}{\prod_{i=1}^{k}\lambda_i }\bigg)^{\frac{1}{n-k}}.$$
\end{theorem}
\begin{proof}
Given that the eigenvalues $\lambda_1,\ldots,\lambda_n$ of $R^{\omeg}(X)$ satisfy $\lambda_1\leq\ldots\leq\lambda_k\leq\lambda_{k+1}\leq\ldots\leq\lambda_n$. Also, $\lambda_1,\ldots,\lambda_k$ are negative and $\lambda_{k+1},\ldots,\lambda_n$ are positive. As $\text{trace}(R^\omeg(X))=0$, we have $\varepsilon(R^\omeg(X))=\sum\limits_{i=1}^{n}|\lambda_i|=2\sum\limits_{i=k+1}^{n} |\lambda_i|=2\sum\limits_{i=1}^{k} |\lambda_i|$. Now
\begin{align*}
|\lambda_1|+|\lambda_2|+\ldots+|\lambda_k|\geq&|\lambda_1+\lambda_2+\ldots+\lambda_k|\\
=&|\lambda_{k+1}+\ldots+\lambda_n|\\
\geq&(n-k)\left(\prod_{i=k+1}^{n}\lambda_i\right)^{\frac{1}{n-k}}\\
=&(n-k) \left(\frac{\de(R^\omeg(X))}{\prod_{i=1}^{k}\lambda_i}\right)^{\frac{1}{n-k}}.
\end{align*}
Hence
\[\varepsilon(R^\omeg(X))=2\sum\limits_{i=1}^{k} |\lambda_i| \geq 2(n-k)\left(\frac{\de(R^\omeg(X))}{\prod_{i=1}^{k}\lambda_i}\right)^{\frac{1}{n-k}}.  \qedhere \]
\end{proof}
\begin{lemma}\label{lem 5.1}
If $x_1, x_2,\ldots, x_n$ are non-negative and $k\geq 2$, then $\sum^n_{i=1}x^k_i\leq \big(\sum^n_{i=1} x^2_i\big)^{k/2}$.
\end{lemma}
Lemma~\ref{lem 5.1} can be easily proved using the principle of mathematical induction and Cauchy-Schwartz inequality.
\begin{theorem}
Let $X$ be a mixed graph. Then $\varepsilon(R^{\omeg}(X))< e^{\sqrt{2R^{(-1)}(X_U)}}.$
\end{theorem}
\begin{proof}
Let the eigenvalues  of $R^{\omeg}(X)$ be $\lambda_1,\ldots,\lambda_n$. We have $\sum_{i=1}^{n}{\lambda_i}^2=2R^{(-1)}(X_U)$. Now
\begin{align*}
\varepsilon(R^\omeg(X))=\sum_{i=1}^{n}|\lambda_i| &< \sum_{i=1}^{n} e^{|\lambda_i|}\\
&=\sum_{i=1}^{n}\sum_{k\geq 0}\frac{|\lambda_i|^k}{k!}\\
&\leq\sum_{k\geq 0}\frac{1}{k!}\left(\sum_{i=1}^{n}{|\lambda_i|^2}\right)^{k/2},
\;\; \text{using Lemma~\ref{lem 5.1}}\\
&\leq\sum_{k\geq 0}\frac{1}{k!}(2R^{(-1)}(X_U))^{k/2}\\
&=\sum_{k\geq 0}\frac{1}{k!}\left(\sqrt{2R^{(-1)}(X_U)}\right)^{k}=e^{\sqrt{2R^{(-1)}(X_U)}}. \qedhere
\end{align*}
\end{proof}
\begin{theorem}
Let $\lambda_1, \ldots,\lambda_n$ be the eigenvalues of $R^\omeg(X)$ and $\rho=\underset{i}{\ma}{|\lambda_i|}$. Then
$$\varepsilon(R^\omeg(X))\leq \frac{1}{2}\left(\rho(n-2)+\sqrt{\rho^2(n-2)^2+16R^{(-1)}(X_U)}\right).$$ \qedhere
\end{theorem}
\begin{proof}
Suppose $\lambda_k$ is the largest negative eigenvalue of $R^\omeg(X)$. Then $\lambda_1\leq\cdots\leq\lambda_k$ and $\lambda_{k+1}\leq\ldots\leq\lambda_n$. 
Now
\begin{align*}
\varepsilon(R^\omeg(X))^2=&\left(\sum_{i=1}^{k}|\lambda_i|+\sum_{j=k+1}^{n}|\lambda_j|\right)^2\\
=&2\left(\left(\sum_{i=1}^{k}|\lambda_i|\right)^2+\left(\sum_{j=k+1}^{n}|\lambda_j|\right)^2 \right),\;\;\text{as} \; \sum_{i=1}^{k}|\lambda_i|=\sum_{j=k+1}^{n}|\lambda_j| \\
=&2\left(\sum_{i=1}^{k}|\lambda_i|^2+\sum_{j=k+1}^{n}|\lambda_j|^2+2\sum_{1\leq i<p\leq k}|\lambda_i||\lambda_p|+2\sum_{(k+1)\leq j<q\leq n}|\lambda_j||\lambda_q|\right)\\
=&2\sum_{i=1}^{n}|\lambda_i|^2+4\left[\sum_{1\leq i<p\leq k}|\lambda_i||\lambda_p|+\sum_{(k+1)\leq j<q\leq n}|\lambda_j||\lambda_q|\right].\numberthis\label{eq 1.5.1}
\end{align*}

We have $\left(|\lambda_i|-\rho/2 \right)\left(|\lambda_p|-\rho/2 \right)\leq \frac{\rho^2}{4}$, which implies that $|\lambda_i||\lambda_p|\leq \frac{\rho}{2}(|\lambda_i|+|\lambda_p|)$. Similarly, $|\lambda_j||\lambda_q|\leq \frac{\rho}{2}(|\lambda_j|+|\lambda_q|)$.

Hence from Equation (\ref{eq 1.5.1}), we have
\begin{align*}
\varepsilon(R^\omeg(X))^2\leq&4R^{(-1)}(X_U)+4\cdot\frac{\rho}{2}\left(\sum_{1\leq i<p\leq k}(|\lambda_i|+|\lambda_p|)+\sum_{(k+1)\leq j<q\leq n}(|\lambda_j|+|\lambda_q|)\right)\\
=&4R^{(-1)}(X_U)+2\rho\left((k-1)\sum_{i=1}^{k}|\lambda_i|+(n-k-1)\sum_{j=k+1}^{n}|\lambda_j|\right)\\
=&4R^{(-1)}(X_U)+2\rho\left((k-1)\frac{\varepsilon(R^\omeg(X))}{2}+(n-k-1)\frac{\varepsilon(R^\omeg(X))}{2}\right)\\
=&4R^{(-1)}(X_U)+\rho(n-2)\varepsilon(R^\omeg(X)).
\end{align*}
After solving the preceding inequality, we get
\[\varepsilon(R^\omeg(X))\leq \frac{1}{2}\left(\rho(n-2)+\sqrt{\rho^2(n-2)^2+16R^{(-1)}(X_U)}\right). \qedhere \]
\end{proof}
\begin{theorem}
Let the eigenvalues of $R^\omeg(X)$ be $\lambda_1,\ldots,\lambda_n$ and $\sigma=\underset{i}{\mi}{|\lambda_i|}$. Then
$$\varepsilon(R^\omeg(X))\geq \frac{1}{2}\left(\sigma(n-2)+\sqrt{\sigma^2(n-2)^2+16R^{(-1)}(X_U)}\right).$$
\end{theorem}
\begin{proof}
Considering $\sigma=\underset{i}{\mi}{|\lambda_i|}$, we have $\left(|\lambda_i|-\sigma/2 \right)\left(|\lambda_p|-\sigma/2 \right)\geq \frac{\sigma^2}{4}$. It implies that $|\lambda_i||\lambda_p|\geq \frac{\sigma}{2}(|\lambda_i|+|\lambda_p|)$. Similarly, $|\lambda_j||\lambda_q|\geq \frac{\sigma}{2}(|\lambda_j|+|\lambda_q|)$. Now from Equation (\ref{eq 1.5.1}), we get the quadratic inequality $\varepsilon(R^\omeg(X))^2\geq 4R^{(-1)}(X_U)+\sigma(n-2)\varepsilon(R^\omeg(X))$. Solving this quadratic inequality, we get the required result.
\end{proof}
We now provide some basic inequalities that help us in determining two other lower bounds for the energy of $R^\omeg(X)$.
\begin{lemma}[P\'olya-Szeg\"o Inequality \cite{fa}]\label{l 1.5.1}
If $a_i$ and $b_i$ are positive real numbers for each $i\in \{1,\ldots,n\}$, then
$$\sum_{i=1}^{n}a_i^2\sum_{i=1}^{n}b_i^2\leq \frac{1}{4}\left(\sqrt{\frac{M_1M_2}{m_1m_2}}+\sqrt{\frac{m_1m_2}{M_1M_2}}\right)^2\left(\sum_{i=1}^{n}a_ib_i \right)^2,$$
where $M_1=\underset{1\leq i\leq n}{\ma}a_i$, $M_2=\underset{1\leq i\leq n}{\ma}b_i$, $m_1=\underset{1\leq i\leq n}{\mi}a_i$ and $m_2=\underset{1\leq i\leq n}{\mi}b_i$.
\end{lemma}
\begin{lemma}[Ozeki's Inequality \cite{fa}]\label{l 1.5.2}
If $a_i$ and $b_i$ are non-negative real numbers for each $i\in \{1,\ldots,n\}$, then
$$\sum_{i=1}^{n}a_i^2\sum_{i=1}^{n}b_i^2-\left(\sum_{i=1}^{n}a_ib_i \right)^2\leq \frac{n^2}{4}\left(M_1M_2-m_1m_2 \right)^2,$$
where $M_1=\underset{1\leq i\leq n}{\ma}a_i$, $M_2=\underset{1\leq i\leq n}{\ma}b_i$, $m_1=\underset{1\leq i\leq n}{\mi}a_i$ and $m_2=\underset{1\leq i\leq n}{\mi}b_i$.
\end{lemma}
\begin{theorem}
Let $\lambda_1,\ldots,\lambda_n$ be the eigenvalues of $R^\omeg(X)$, $\rho=\underset{i}{\ma}{|\lambda_i|}$ and $\sigma=\underset{i}{\mi}{|\lambda_i|}$. Then
$$\varepsilon(R^\omeg(X))\geq \frac{\sqrt{8n\rho\sigma R^{(-1)}(X_U)}}{\rho+\sigma}.$$
\end{theorem}
\begin{proof}
Considering $\rho=\underset{i}{\ma}{|\lambda_i|}$ and $\sigma=\underset{i}{\mi}{|\lambda_i|}$, by P\'olya-Szeg\"o Inequality we have
\begin{align*}
&\sum_{i=1}^{n}|\lambda_i|^2\sum_{i=1}^{n}1^2\leq \frac{1}{4}\left(\sqrt{\frac{\rho}{\sigma}}+\sqrt{\frac{\sigma}{\rho}}\right)^2\left(\sum_{i=1}^{n}|\lambda_i|\right)^2\\
\text{or,} \;\;\;\;\; &2n R^{(-1)}(X_U) \leq \frac{1}{4}\left(\frac{\rho+\sigma}{\sqrt{\sigma\rho}}\right)^2\left(\varepsilon(R^\omeg(X))\right)^2\\
\text{or,} \;\;\;\;\; &\varepsilon(R^\omeg(X))\geq \frac{\sqrt{8n\rho\sigma R^{(-1)}(X_U)}}{\rho+\sigma}. \qedhere
\end{align*}
\end{proof}
\begin{theorem}
Let $\lambda_1,\ldots,\lambda_n$ be the eigenvalues of $R^\omeg(X)$, $\rho=\underset{i}{\ma}{|\lambda_i|}$ and $\sigma=\underset{i}{\mi}{|\lambda_i|}$. Then
$$\varepsilon(R^\omeg(X))\geq \frac{\sqrt{8nR^{(-1)}(X_U)-n^2(\rho-\sigma)^2}}{2}.$$
\end{theorem}
\begin{proof}
Considering $\rho=\underset{i}{\ma}{|\lambda_i|}$ and $\sigma=\underset{i}{\mi}{|\lambda_i|}$, by Ozeki's Inequality we have
\begin{align*}
&\sum_{i=1}^{n}|\lambda_i|^2\sum_{i=1}^{n}1^2-\left(\sum_{i=1}^{n}|\lambda_i|\right)^2\leq \frac{n^2}{4}\left(\rho-\sigma \right)^2\\
\text{or,} \;\;\;\;\;\; &\frac{\sqrt{8nR^{(-1)}(X_U)-n^2(\rho-\sigma)^2}}{2} \leq \varepsilon(R^\omeg(X)). \qedhere
\end{align*}
\end{proof}
\section*{Funding}
The first two authors thank the Science and Engineering Research Board (SERB), Government of India for supporting major part of this work under the Teachers Associateship for Research Excellence (TARE) project [file number TAR/2021/000045].

\end{document}